\documentclass[11pt]{article}
\usepackage{amsthm,amstext}
\usepackage{amsmath}
\usepackage{graphicx}
\usepackage{amsfonts}
\usepackage{amssymb}

\theoremstyle{plain}
\newtheorem{theorem}{Theorem}[section]

\newtheorem{proposition}[theorem]{Proposition}

\theoremstyle{definition}
\newtheorem{definition}[theorem]{Definition}
\newtheorem{corollary}[theorem]{Corollary}
\newtheorem{example}{\sc Example}
\theoremstyle{remark}
\newtheorem{remark}{\sc Remark}

\vspace{.25 in}

\begin{document}

\date{}
\title{\textbf{$(\alpha, \beta)$-Fuzzy Subsemigroup
and $(\alpha,\beta)$-Fuzzy Bi-ideal in $\Gamma$- Semigroups}}
\author{ \textbf{Sujit Kumar Sardar$^1$}, \textbf{Bijan Davvaz$^2$},
\textbf{Samit Kumar Majumder$^3$}\\ and\\ \textbf{Soumitra Kayal$^{4}$}\\\\
$^{1,4}$Department of Mathematics, Jadavpur\\
University, Kolkata-700032, INDIA\\
$^2$Department of Mathematics, Yazd University,\\ Yazd, IRAN \\
$^3$Tarangapur N.K High School, Tarangapur,\\ Uttar Dinajpur, West Bengal-733129, INDIA\\
\texttt{$^1$sksardarjumath@gmail.com}\\
\texttt{$^2$davvaz@yazduni.ac.ir}\\
\texttt{$^3$samitfuzzy@gmail.com}\\
\texttt{$^4$soumitrakayal.ju@gmail.com} } \maketitle

\begin{abstract}
In this paper, using the idea of quasi-coincidence of a fuzzy point with a fuzzy set, the concepts of $(\in ,\in \vee q)$- fuzzy subsemigroup and $(\in ,\in \vee q)$- fuzzy bi-ideal in a $\Gamma$-semigroup have been introduced and some related properties have been investigated.

\textbf{AMS Mathematics Subject Classification[2000]:}\textit{\ }08A72,
20M12, 3F55

\textbf{Key Words and Phrases:}\textit{\ }$\Gamma $-semigroup, belong to or quasi-coincident, $(\in ,\in\vee q)$-fuzzy subsemigroup, $(\in ,\in \vee q)$-fuzzy bi-ideals, $(\alpha,\beta)$-fuzzy subsemigroup, $(\alpha,\beta)$-fuzzy bi-ideal.
\end{abstract}


\section{Introduction}

The concept of fuzzy set was introduced by Zadeh\cite{Z}. Since then many researchers explored on the generalizations of fuzzy sets. Many papers on fuzzy sets appeared showing the importance of the concept and its application to logic, set theory, group theory, semigroup theory, real analysis, measure theory, topology etc. It was first applied to the theory of groups by A. Rosenfeld\cite{R}. In \cite{Yu}, Yuan et al. introduced the definition of fuzzy subgroup with thresholds which is a generalization of Rosenfeld's fuzzy subgroup and Bhakat and Das's fuzzy subgroup. Murali\cite{Mu} proposed a definition of fuzzy point belonging to fuzzy subset under a natural equivalence on fuzzy subset. the idea of quasi coincidence of a fuzzy point with a fuzzy set, which is mentioned in \cite{Pu}, played a vital role to generate some different types of fuzzy subgroups. Bhakat and Das\cite{Bh1, Bh2} gave the concept of $(\alpha,\beta)$-fuzzy subgroups by using the {\bf belong to} relation $(\in)$ and {\bf quasi-coincidence with} relation $(q)$ between a fuzzy point and a fuzzy subgroup, and introduced the concept of an $(\in ,\in \vee q)$- fuzzy subgroup. In particular, $(\in,\in\vee q)$- fuzzy subgroup is an important and useful generalization of Rosenfeld's fuzzy subgroup. We see the
fuzzification of different concepts of semigroups in \cite{K1,K2,K3}. Yunqiang Yin and Dehua Xu\cite{Yu1} introduced the concepts of $(\in,\in \vee q)$- fuzzy subgroup and $(\in,\in \vee q)$- fuzzy ideals in semigroups. In \cite{J}, Y.B. Jun, S.Z. Song introduced the notion of generalized fuzzy interior ideals in semigroups. Sen and Saha in \cite{SS} defined the concepts of $\Gamma$-semigroups as a generalization of semigroups. $\Gamma$-semigroups have been analyzed by a lot of mathematicians, for instance Chattopadhay\cite{Ch1,Ch2,SC}, Dutta and Adhikari\cite{D1,D3}, Hila\cite{H1,H2}, Chinram\cite{Chin}, Saha\cite{S,SSa}, Seth\cite{Se}. Sardar and Majumder\cite{S1,S3,S5} characterized subsemigroups, bi-ideals, interior ideals(along with B.Davvaz\cite{S4}), quasi ideals, ideals, prime(along with D. Mandal\cite{S2}) and semiprime ideals, ideal extensions(along with T.K. Dutta\cite{D2,D4}) of a $\Gamma$-semigroup in terms of fuzzy subsets. They also studied their different properties directly and via operator semigroups of a $\Gamma$-semigroup. As a first step in this direction, here the authors are going to introduce the concept of $(\in,\in \vee q)$- fuzzy subsemigroup and $(\in,\in \vee q)$- fuzzy bi-ideal in a $\Gamma$-semigroup.


\section{Preliminaries}

In this section we discuss some elementary definitions which will be used in the sequel.

Let $S=\{x,y,z,....\}$ and $\Gamma=\{\alpha,\beta,\gamma,....\}$ be two non-empty sets. Then $S$ is called a $\Gamma$-{\it semigroup}$\cite{SS}$ if there exist a mapping $S\times\Gamma\times S\rightarrow S($images to be denoted by $a\alpha b)$ satisfying

\begin{itemize}
\item[(1)] $x\gamma y\in S$,

\item[(2)] $(x\beta y)\gamma z=x\beta (y\gamma z)$, for all $x,y,z\in S$ and for all $\beta,\gamma\in\Gamma.$
\end{itemize}

A non-empty subset $A$ of a $\Gamma $-semigroup $S$ is called a {\it subsemigroup}\cite{S5} of $S$ if $A\Gamma A\subseteq A.$ A subsemigroup $A$ of a $\Gamma $-semigroup $S$ is called a {\it bi-ideal}\cite{S5} of $S$ if $A\Gamma S\Gamma A\subseteq A.$

A function $\mu$ from a non-empty set $X$ to the unit interval $[0,1]$ is called a {\it fuzzy subset}\cite{Z} of $X.$

A $\Gamma$-semigroup $S$ is called {\it regular}\cite{D1}, if for each $a\in S,$ there exist $x\in S$ and $\alpha,\beta\in\Gamma$ such that $a=a\alpha x\beta a.$

A $\Gamma$-semigroup $S$ is called {\it intra-regular}\cite{D1}, if for each $a\in S,$ there exists $x,y\in S$ and
$\alpha,\beta,\gamma\in\Gamma$ such that $a=x\alpha a\beta a\gamma y.$

A $\Gamma$-semigroup $S$ is called left$($right$)$ duo if every left$($resp. right$)$ ideal of $S$ is a two sided ideal of $S.$

A $\Gamma$-semigroup $S$ is called duo if it is left and right duo.

\begin{example}
\cite{S4} Let $\Gamma=\{5,7\}.$ For any $x,y\in N$ and $\gamma\in\Gamma ,$ define $x\gamma y=x.\gamma .y$ where $.$ is the usual multiplication on $N$. Then $N$ is a $\Gamma$-semigroup.
\end{example}

\begin{definition}
Let $S$ be a $\Gamma$-semigroup. For a fuzzy subset $\mu$ of $S$ and $t\in (0,1],$ the set $U(A;t)=\{x\in S:\mu(x)\geq t\}$ is called a level subset of $S$ determined by $\mu$ and $t.$
\end{definition}

\begin{definition}
\cite{Mu} A fuzzy subset $\mu$ of a set $X$ of the form
\begin{equation*}
\mu (y)=\left\{
\begin{array}{c}
t(\neq 0)\text{ \ if }y=x \\
0\text{ \ \ \ \ \ \ \ if }y\neq x%
\end{array}%
\right.
\end{equation*}

is said to be a fuzzy point with support $x$ and value $t$ and is denoted by $x_{t}.$
\end{definition}

\begin{definition}
\cite{S5} $(S1)$ A non-empty fuzzy subset $\mu$ of a $\Gamma$-semigroup $S$ is called a fuzzy subsemigroup of $S$ if $\mu(x\gamma y)\geq\min\{\mu(x),\mu(y)\} \forall x,y\in S,\forall\gamma\in\Gamma.$
\end{definition}

\begin{definition}
\cite{S5} $(B1)$ A fuzzy subsemigroup $\mu$ of a $\Gamma$-semigroup $S$ is called a fuzzy bi-ideal of $S$ if $\mu(x\alpha y\beta z)\geq\min\{\mu(x),\mu(z)\} $ $\forall x,y,z\in S,\forall\alpha,\beta\in\Gamma.$
\end{definition}

\begin{definition}
\cite{Yu1} A fuzzy point $x_{t}$ is said to belong to $($be quasi coincident with$)$ a fuzzy subset $\mu,$ written as $x_{t}\in\mu($resp. $x_{t}q\mu)$ if $\mu(x)\geq t($resp. $\mu(x)+t>1).$

$x_{t}\in\mu$ or $x_{t}q\mu$ will be denoted by $x_{t}\in\vee q\mu,$ $x_{t}\in\mu$ and $x_{t}q\mu$ will be denoted by
$x_{t}\in\wedge q\mu.$ $x_{t}\overline{\in}\mu,$ $x_{t}\overline{\in\vee q}\mu$ and $x_{t}\overline{\in\wedge q}\mu$ will respectively mean $x_{t}\notin \mu,x_{t}\notin \vee q\mu$ and $x_{t}\notin\wedge q\mu$.
\end{definition}

\begin{definition}
\cite{Yu1} Let $X$ be a non-empty set and $\mu$ be a fuzzy subset of $X.$ Then for any $t\in (0,1],$ the sets $\mu _{t}=\{x\in X:\mu (x)\geq t\}$ and $Supp(\mu)=\{x\in X:\mu(x)>0\}$ are called $t$-level subset and supporting set of $\mu$ respectively.
\end{definition}

\begin{definition}
\cite{S1}Let $S$ be a $\Gamma$-semigroup, $\lambda$ and $\mu$ are fuzzy subsets of $S.$ Then the product of $\lambda$ and $\mu,$ denoted by, $\lambda\circ\mu $ is defined by
\begin{equation*}
(\lambda \circ \mu )(x)=\left\{
\begin{array}{c}
\underset{x=y\gamma z}{\sup }\min \{\lambda (y),\mu (z)\}\text{ }\forall
y,z\in S,\forall \gamma \in \Gamma \\
0\text{ \ \ \ \ \ \ \ \ \ \ \ otherwise}%
\end{array}%
\right. .
\end{equation*}

Also, for any fuzzy subsets $\lambda,\mu$ and $\nu$ of $S,(\lambda\circ\mu)\circ\nu=\lambda\circ(\mu\circ\nu).$
\end{definition}


\section{Main Results}

\begin{definition}
$(1)$ A non-empty fuzzy subset $\mu$ of a $\Gamma$-semigroup $S$ is said to be an $(\in,\in\vee q)$-fuzzy subsemigroup of $S$ if $\forall x,y\in S,\forall\gamma\in\Gamma$ and $t,r\in (0,1],x_{t},y_{r}\in\mu\Rightarrow(x\gamma y)_{\min(t,r)}\in\vee q\mu.$

$(2)$ An $(\in,\in\vee q)$-fuzzy subsemigroup $\mu$ of a $\Gamma$-semigroup $S$ is said to be an $(\in,\in\vee q)$-fuzzy bi-ideal of $S$ if $\forall x,y,z\in S,\forall\alpha,\beta\in\Gamma$ and $t,r\in (0,1], x_{t},z_{r}\in\mu\Rightarrow(x\alpha y\beta z)_{\min(t,r)}\in\vee q\mu.$
\end{definition}

\begin{theorem}
Let $\mu$ be any non-empty fuzzy subset of a $\Gamma$-semigroup $S.$ Then the following statements are equivalent: $(1)$ $\mu$ is an $(\in,\in \vee q)$-fuzzy subsemigroup of $S,$ $(2)$ for any $x,y\in Supp(\mu),\gamma\in\Gamma,\mu (x\gamma y)\geq\min \{\mu (x),\mu (y),0.5\},$ $(3)$ $\mu \circ \mu \subseteq \vee q\mu ,$ $(4)$ $\mu \circ \mu \cap $ $0.5_{Supp(\mu)}\subseteq \mu ,$ $(5)$ for any $r\in (0,0.5],$ if $\mu _{r}$ is non-empty, then
$\mu _{r}$ is a subsemigroup of $S.$
\end{theorem}

\begin{proof}
$(1)\Rightarrow(2):$ Let $\mu$ be an $(\in, \in \vee q)$-fuzzy subsemigroup of $S$. Let $x,y\in Supp(\mu), \gamma\in \Gamma$. If possible, let $\mu(x\gamma y)<\min \{\mu(x),\mu(y),0.5\}$. Let us choose $r,t\in (0,1]$ so that $\mu(x\gamma y)<\min(r,t)<\min\{\mu(x),\mu(y),0.5\}.$ Then $x_{\min(r,t)},y_{\min(r,t)}\in\mu$ but $(x\gamma y)_{\min(r,t)}\overline{\in \vee q}\mu,$ a contradiction. Therefore, $(1)\Rightarrow(2).$\\

$(2)\Rightarrow(3):$ Let us suppose that $(2)$ holds. Let $x_{r}\in \mu\circ\mu$. Then$(\mu \circ \mu)(x)\geq r.$ If possible, let $x_{r}\overline{\in \vee q}\mu.$ Then $\mu(x)<r $ and $\mu(x)+r\leq 1.$ Hence $\mu(x)<0.5.$ Let $x=y\gamma z$ for $y,z\in S, \gamma\in\Gamma$. Then $\mu(x)=\mu(y\gamma z)\geq\min\{\mu(y),\mu(z),0.5\}=\min\{\mu(y),\mu(z)\}.$ Now $r\leq(\mu\circ\mu)(x)=\underset{x=y\gamma z}{\sup }\min \{\mu (y),\mu(z)\}\leq\underset{x=y\gamma z}{\sup }\mu(y\gamma z)=\mu(x),$ a contradiction. Hence $x_{r}\in \vee q\mu.$ Therefore $(2)\Rightarrow(3).$\\

$(3)\Rightarrow(4):$ Let $(3)$ holds. Let $x_{r}\in\mu \circ \mu \cap 0.5_{Supp(\mu)}.$ Then $x_{r}\in(\mu\circ\mu)$ and $r\leq 0.5$. If possible let $x_{r}\notin\mu,$ then $\mu(x)<r\leq0.5$. Hence $\mu(x)<r$ and $\mu(x)+r<1$ $\Rightarrow x_{r}\overline{\in \vee q}\mu,$ which contradicts $x_{r}\in(\mu\circ\mu)\subseteq \vee q\mu.$ Hence $x_{r}\in\mu.$ Therefore, $(3)\Rightarrow(4).$\\

$(4)\Rightarrow(5):$ Let $(4)$ holds. Let $r\in(0,0.5]$ such that $\mu_{r}$ is nonempty. Let $x,y\in\mu_{r}$ and $\gamma\in\Gamma.$ Then $\mu(x)\geq r,\mu(y)\geq r.$ Then $\mu(x\gamma y)\geq(\mu \circ \mu \cap0.5_{Supp(\mu)})(x\gamma y)=\min\{(\mu\circ\mu)(x\gamma y),0.5\}\geq\min \{\min\{\mu (x),\mu(y)\},0.5\}
\geq\min\{r,0.5\}=r.$ This implies that $x\gamma y\in\mu_{r}.$ Hence $\mu_{r}$ is a subsemigroup of S. Therefore $(4)\Rightarrow(5).$\\

$(5)\Rightarrow(1):$ Let $(5)$ holds. Let $x,y\in S$, $\gamma\in\Gamma$ and $t,r\in (0,1]$ such that $x_{t},y_{r}\in\mu$. Then $\mu(x)\geq t,\mu(y)\geq r$. If possible let $(x\gamma y)_{\min(t,r)}\overline{\in \vee
q}\mu$. Then $\mu(x\gamma y)<\min(t,r)$ and $\mu(x\gamma y)+\min(t,r)\leq 1.$ Hence $\mu(x\gamma y)<0.5.$ Let us choose $s$ such that $\mu(x\gamma y)<s \leq\min(r,t,0.5)\leq\min\{\mu(x),\mu(y),0.5\}.$ Then $x,y\in \mu_{s}$ but $x\gamma
y\notin\mu_{s},$ a contradiction. Therefore, $(5)\Rightarrow(1).$
\end{proof}

\begin{remark}
$(1)$ It is clear that if $\mu $ is an $(\in ,\in \vee q)$-fuzzy subsemigroup of $S,$ then $Supp(\mu )$ is a subsemigroup of $S.$ $(2)$ If $\mu $ is a fuzzy subsemigroup of a $\Gamma $-semigroup $S,$ then $\mu $ is an $(\in, \in \vee q)$-fuzzy subsemigroup of $S.$ However, the converse is not necessarily true which is clear from the following example.
\end{remark}

\begin{example}
Let $S=\{e,a,b\}$ and $\Gamma =\{\gamma \}$, where $\gamma $ is defined on $%
S $ with the following cayley table:
\begin{equation*}
\begin{array}{l|lll}
\gamma & e & a & b \\ \hline
e & e & e & e \\
a & e & a & e \\
b & e & e & b%
\end{array}%
\end{equation*}

Then $S$ is a $\Gamma $-semigroup. We define a fuzzy subset $\mu:S\rightarrow \lbrack 0,1]$ as
\begin{equation*}
\mu (x)=\left\{%
\begin{array}{l}
0.5\text{ \ if }x=e \\
0.6\text{ if }x=a,b.%
\end{array}%
\right. .
\end{equation*}

Then it is easy to verify that $\mu $ is an $(\in, \in \vee q)$-fuzzy subsemigroup of $S,$ but it is not a fuzzy subsemigroup of $S.$
\end{example}

\begin{theorem}
Let $\mu $ be any $(\in, \in \vee q)$-fuzzy subsemigroup of a $\Gamma$-semigroup $S.$ Then the following statements are equivalent: $(1)$ $\mu $ is an $(\in, \in \vee q)$-fuzzy bi-ideal of $S,$ $(2)$ for any $x,z\in Supp(\mu),y\in S$ and $\alpha,\beta \in \Gamma ,\mu(x\alpha y\beta z)\geq \min \{\mu (x),\mu (z),0.5\},$ $(3)$ $\mu \circ\chi_{S}\circ \mu \subseteq \vee q\mu($where $\chi_{S}$ is the characteristic function of $S),$ $(4)$ $\mu \circ\chi_{S}\circ \mu \cap $ $0.5_{Supp(\mu)}\subseteq\mu($where $\chi_{S}$ is the characteristic function of $S),$ $(5)$ for any $r\in (0,0.5],$ if $\mu _{r}$ is non-empty, then $\mu _{r}$ is a bi-ideal of $S.$
\end{theorem}

\begin{proof}
$(1)\Rightarrow(2):$ Let $\mu$ be an $(\in,\in \vee q)$-fuzzy bi-ideal of $S.$ Let $x,z\in Supp(\mu),\alpha,\beta\in \Gamma$ and y$\in $S be such that $\mu(x\alpha y\beta z)<\min\{\mu(x),\mu(z),0.5\}$. Let us consider $r,t$ so that $\mu(x\alpha y\beta z)<\min(r,t)<\min\{\mu(x),\mu(z),0.5\}.$ Then $(x\alpha y\beta z)_{\min(r,t)}\overline{\in \vee q}\mu,$ though $x_{\min(r,t)},z_{\min(r,t)}\in\mu,$ a contradiction. Therefore, $(1)\Rightarrow(2).$\\

$(2)\Rightarrow(3):$ Let $(2)$ holds. Let $x_{r}\in \mu\circ \chi_{S}\circ \mu$, then $(\mu\circ\chi_{S}\circ\mu)(x)\geq r.$ If possible let $x_{r}\overline{\in \vee q}\mu.$ Then $\mu(x)<r $ and $\mu(x)+r\leq 1.$ Hence $\mu(x)<0.5.$ Let $x=p\alpha z$ where $z=q\beta r$ for $p,z,q,r\in S,\alpha,\beta\in\Gamma$. Then $\mu(x)=\mu( p\alpha z)=\mu(p\alpha q\beta r)\geq\in\{\mu(p),\mu(r),0.5\}=\min\{\mu(p),\mu(r)\}.$ Now $r\leq
(\mu\circ\chi_{S}\circ\mu)(x)=\underset{x=p\alpha z}{\sup}\min\{\mu(p),(\chi_{S}\circ\mu)(z)\}=\underset{x=p\alpha
z}{\sup}\min\{\mu(p),\underset{z=q\beta r}{\sup}\min\{\chi_{S}(q),\mu(r)\}\}\leq\underset{x=p\alpha q\beta r}{\sup}\min\{\mu(p),\mu(r)\}\}=\mu(x),$ a contradiction. Hence $x_{r}\in \vee q\mu.$ Therefore, $(2)\Rightarrow(3).$\\

$(3)\Rightarrow(4):$ Let $(3)$ holds. Let $x_{r}\in\mu \circ \chi_{S}\circ\mu \cap 0.5_{Supp(\mu)}.$ Then $x_{r}\in(\mu\circ\chi_{S}\circ\mu)$ and $r\leq 0.5$. If $x_{r}\notin\mu$, then $\mu(x)<r\leq0.5$. Hence $\mu(x)<r$ and $\mu(x)+r\leq 1$. Then $x_{r}\overline{\in \vee q}\mu,$ which contradicts $x_{r}\in(\mu\circ\chi_{S}\circ\mu)\subseteq\vee q\mu.$ Hence $x_{r}\in\mu.$ Therefore, $(3)\Rightarrow(4).$\\

$(4)\Rightarrow(5):$ Let $(4)$ holds. Let $r\in(0,0.5]$ and $\mu_{r}$ is nonempty. Since $\mu$ is an $(\in, \in \vee q)$-fuzzy subsemigroup of $S,$ then by Theorem $3.2,$ $\mu_{r}$ is a subsemigroup of $S.$ Let $x,z\in\mu_{r}$, $y\in S,\alpha,\beta\in\Gamma$. Then $\mu(x),\mu(z)\geq r$. Now $\mu(x\alpha y\beta z)\geq((\mu\circ\chi_{S}\circ\mu)\cap 0.5_{Supp(\mu)})(x\alpha y\beta z)=\min\{(\mu\circ\chi_{S}\circ\mu)(x\alpha y\beta z),0.5\}\geq\min\{\min\{\mu(x),\chi_{S}(y),\mu(z)\},0.5\}\\=\min\{\mu(x),\mu(z),0.5\}\geq\min\{r,r,0.5\}=r$. Therefore $x\alpha y\beta z\in\mu_{r}$. Hence $\mu_{r}$ is a bi-ideal of $S.$ Hence $(4)\Rightarrow(5).$\\

$(5)\Rightarrow(1):$ Let $(5)$ holds. Let $x,y,z\in S,\alpha,\beta\in\Gamma$ and $t,r\in(o,1]$ such that $x_{t},z_{r}\in\mu$. If possible let $(x\alpha y\beta z)_{\min(t,r)}\overline{\in \vee q}\mu$. Then $\mu(x\alpha y\beta z)<\min(t,r)$ and $\mu(x\alpha y\beta z)+\min(t,r)\leq 1.$ Hence $\mu(x\alpha y\beta z)<0.5.$ Let us choose $p$ such that $\mu(x\alpha y\beta z)<p \leq\min(t,r,0.5)\leq\min\{\mu(x),\mu(z),0.5\}.$ Then $x,z\in \mu_{p}$ but $x\alpha y\beta z\notin\mu_{p},$ a contradiction. Therefore, $(5)\Rightarrow(1).$
\end{proof}

\begin{corollary}
Let $\mu $ is an $(\in, \in \vee q)$-fuzzy subsemigroup$((\in, \in  \vee q)$-fuzzy bi-ideal$)$ of $S.$ Then $Supp(\mu)$ is a subsemigroup$($resp. bi-ideal$)$ of $S.$
\end{corollary}

\begin{theorem}
Let \{$\mu_{i}:i\in I$\} be any family of $(\in,\in \vee q)$-fuzzy subsemigroups of $S.$ Then $\bigcap_{i\in I}\mu_{i}$ and $\bigcup_{i\in I}\mu_{i}$ is a $(\in,\in \vee q)$-fuzzy subsemigroup of $S.$ If \{$\mu_{i}:i\in I$\} be any family of $(\in,\in \vee q)$-fuzzy bi-ideals of $S,$ then both $\bigcap_{i\in I}\mu_{i}$ and $\bigcup_{i\in I}\mu_{i}$ are $(\in, \in \vee q)$-fuzzy bi-ideals of $S.$
\end{theorem}

\begin{proof}
Let $\{\mu_{i}:i\in I\}$ be any family of $(\in, \in \vee q)$-fuzzy subsemigroup of $S.$ Let $\theta=\bigcap_{i\in I}\mu_{i}$. Let $x,y\in Supp(\theta),\gamma\in\Gamma$. Then
\begin{align*}
\theta(x\gamma y)&=\underset{i\in I}{\min}\mu_{i}(x\gamma y)\\
&\geq\underset{i\in I}{\min}\min\{\mu_{i}(x),\mu_{i}(y),0.5\}\\
&=\min\{\underset{i\in I}{\min}\mu_{i}(x),\underset{i\in I}{\min}\mu_{i}(y),0.5\}\\
&=\min\{\theta(x),\theta(y),0.5\}.
\end{align*}
Therefore $\theta=\bigcap_{i\in I}\mu_{i}$ is an $(\in, \in \vee q)$-fuzzy subsemigroup of $S.$ Similarly $\bigcup_{i\in I}\mu_{i}$ is an $(\in, \in \vee q)$-fuzzy subsemigroup of $S.$\\

Let $\{\mu_{i}:i\in I\}$ be any family of $(\in, \in \vee q)$-fuzzy bi-ideals of $S.$ Let $\theta=\bigcap_{i\in I}\mu_{i}$. Then $\theta$ is a $(\in,\in \vee q)$-fuzzy subsemigroup of $S.$ Let $x,z\in Supp(\theta),y\in S,\alpha,\beta\in\Gamma$. Then
\begin{align*}
\theta(x\alpha y\beta z)&=\underset{i\in I}{\min}\mu_{i}(x\alpha y\beta z)\\
&\geq\underset{i\in I}{\min}\min\{\mu_{i}(x),\mu_{i}(z),0.5\}\\
&\geq\min\{\underset{i\in I}{\min}\mu_{i}(x),\underset{i\in I}{\min}\mu_{i}(z),0.5\}\\
&=\min\{\theta(x),\theta(z),0.5\}.
\end{align*}
Therefore $\theta=\bigcap_{i\in I}\mu_{i}$ is an $(\in, \in \vee q)$-fuzzy bi-ideals of $S.$ Similarly $\bigcup_{i\in I}\mu_{i}$ is an $(\in, \in \vee q)$-fuzzy bi-ideals of $S.$
\end{proof}

In view of Theorem $3.5$ we can have the following theorems.

\begin{theorem}
The family of all the $(\in, \in \vee q)$-fuzzy subsemigroups of $S$ with fuzzy set inclusion relation $\subseteq$ constitutes a complete lattice. For any $(\in, \in \vee q)$-fuzzy subsemigroups $\mu$ and $\nu$ of $S, \mu\cap\nu$ and $\mu\cup\nu$ are the greatest lower bound and least upper bound of $\{\mu,\nu\},$ respectively. Moreover, it is closed under fuzzy set union and intersection.
\end{theorem}

The following theorem shows that the image and the inverse image of the $(\in,\in \vee q)$-fuzzy subsemigroup and $(\in,\in \vee q)$-fuzzy bi-ideal of a $\Gamma$-semigroup are also $(\in,\in \vee q)$-fuzzy subsemigroup and $(\in,\in \vee q)$-fuzzy bi-ideal.

\begin{definition}
\cite{R} Let $f$ be any function from a set $X$ to a set $X^{'}.$ A fuzzy subset $\mu$ of $X$ is called $f$-invariant if for any $x,y\in X,f(x)=f(y)\Rightarrow\mu(x)=\mu(y).$
\end{definition}

\begin{theorem}
Let $S$ and $S^{'}$ be two $\Gamma$-semigroups, $\mu$ and $\mu^{'}$ are $(\in ,\in \vee q)$-fuzzy subsemigroups of $S$ and $S^{'}$ respectively, and $f$ be a homomorphism from $S$ onto $S^{'}.$ Then $(1)$ $f(\mu)$ is an $(\in, \in \vee q)$-fuzzy subsemigroup of $S^{'}.$ $(2)$ $f^{-1}(\mu^{'})$ is an $(\in ,\in \vee q)$-fuzzy subsemigroup of $S,$ $(3)$ The mapping $\mu\rightarrow f(\mu)$ defines a one-one correspondence between the set of the $f$-invariant $(\in, \in\vee q)$-fuzzy subsemigroups of $S$ and the set of the $(\in, \in\vee q)$-fuzzy subsemigroups of $S^{'}.$
\end{theorem}

\begin{proof}
$(1)$ Let $x^{'},y^{'}\in Supp(f(\mu))$ and $\alpha\in\Gamma$. Then $f(\mu)(x^{'})>0$ and $f(\mu)(y^{'})>0 \Rightarrow \underset{x\in f^{-1}(x^{'})}{\sup}\mu(x)>0$ and $\underset{y\in f^{-1}(y^{'})}{\sup}\mu(y)>0$. So there exists $x,y\in S$ such that $f(x)=x^{'},f(y)=y^{'}$ and $\mu(x)>0,\mu(y)>0$. Then $f(x\alpha y)=x^{'}\alpha y^{'}$. Then

\begin{align*}
f(\mu)(x^{'}\alpha y^{'})&=\underset{z\in f^{-1}(x^{'}\alpha y^{'})}{\sup}\mu(z)\\
&\geq\underset{x\in f^{-1}(x^{'}),y\in f^{-1}(y^{'})}{\sup}\mu(x\alpha y)\\
&\geq\underset{x\in f^{-1}(x^{'}),y\in f^{-1}(y^{'})}{\sup}\min\{\mu(x),\mu(y),0.5\}\\
&=\min\{\underset{x\in f^{-1}(x^{'})}{\sup}\mu(x),\underset{y\in f^{-1}(y^{'})}{\sup}
\mu(y),0.5\}\\
&=\min\{f(\mu)(x^{'}),f(\mu)(y^{'}),0.5\}.
\end{align*}
Hence $f(\mu)$ is an $(\in ,\in \vee q)$-fuzzy subsemigroup of $S^{'}.$\\

$(2)$ Let $x,y\in Supp f^{-1}(\mu^{'})$ and $\alpha\in\Gamma$. Then $f^{-1}(x)>0, f^{-1}(y)>0 \Rightarrow \mu^{'}(f(x))>0,\mu^{'}(f(y))>0$. Then
\begin{align*}
f^{-1}(\mu^{'})(x\alpha y)&=\mu^{'}(f(x\alpha y))=\mu^{'}(f(x)\alpha f(y))\\
&\geq\min\{\mu^{'}(f(x)),\mu^{'}(f(y)),0.5\}\\
&=\min\{f^{-1}(\mu^{'})(x),f^{-1}(\mu^{'})(y),0.5\}.
\end{align*}
Hence $f^{-1}(\mu^{'})$ is an $(\in ,\in \vee q)$-fuzzy subsemigroup of $S.$\\

$(3)$ Using $(1)$ and $(2)$ we have the proof.
\end{proof}

Using similar argument as in Theorem $3.8$ we can have the following Theorem.

\begin{theorem}
Let $S$ and $S^{'}$ be two $\Gamma$-semigroups, $\mu$ and $\mu^{'}$ are $(\in ,\in \vee q)$-fuzzy bi-ideals of $S$ and $S^{'}$ respectively, and $f$ be a homomorphism from $S$ onto $S^{'}.$ Then $(1)$ $f(\mu)$ is an $(\in, \in \vee q)$-fuzzy bi-ideal of $S^{'},$ $(2)$ $f^{-1}(\mu^{'})$ is an $(\in, \in \vee q)$-fuzzy bi-ideal
of $S,$ $(3)$ the mapping $\mu\rightarrow f(\mu)$ defines a one-one correspondence between the set of the $f$-invariant $(\in ,\in\vee q)$-fuzzy bi-ideals of $S$ and the set of the $(\in, \in\vee q)$-fuzzy bi-ideals of $S^{'}.$
\end{theorem}


\section{$(\alpha,\beta)$-Fuzzy Subsemigroup and $(\alpha,\beta)$-Fuzzy Bi-ideal}

In the following $\alpha,\beta$ denote any one of $\in,\in\vee q,\in \wedge q$ unless or otherwise mentioned.

In the following theorem it is shown that every fuzzy subsemigroup of a $\Gamma$-semigroup $S$ is an $(\in,\in)$-subsemigroup of $S.$

\begin{theorem}
For any fuzzy subset $\mu$ of a $\Gamma$-semigroup $S,$ the condition $(S1)$ in Definition $2.3$ is equivalent with the condition $(S2) \forall x,y\in S,\forall\gamma\in\Gamma,t,r\in(0,1],x_{t},y_{r}\in\mu\Rightarrow(x\gamma y)_{min(t,r)}\in\mu.$
\end{theorem}

\begin{proof} $(S1)\Rightarrow(S2):$ Let $(S1)$ holds. Let $x,y\in S,\gamma\in\Gamma$ and $t,r\in(0,1]$ be such that $x_{t},y_{r}\in\mu.$ Then $\mu(x)\geq t$ and $\mu(y)\geq r.$ So, by $(S1)$ we have $\mu(x\gamma y)\geq\min\{\mu(x),\mu(y)\}\geq\min\{t,r\}.$ Hence $(x\gamma y)_{\min(t,r)}\in\mu.$

$(S2)\Rightarrow(S1):$ Let $(S2)$ holds. Let $x,y\in S,\gamma\in\Gamma.$ Since $x_{\mu(x)},y_{\mu(y)}\in\mu,$ so by $(S2)$, we have $(x\gamma y)_{\min\{\mu(x),\mu(y)\}}\in\mu.$ Consequently, $\mu(x\gamma y)\geq\min\{\mu(x),\mu(y)\}.$ Hence the proof.
\end{proof}

In the following theorem it is shown that every fuzzy bi-ideal of a $\Gamma$-semigroup $S$ is an $(\in,\in)$-bi-ideal of $S.$

\begin{theorem}
For any fuzzy subset $\mu$ of a $\Gamma$-semigroup $S,$ the condition $(B1)$ in Definition $2.4$ is equivalent with the condition $(B2) \forall x,y,z\in S,\forall\gamma,\delta\in\Gamma,t,r\in(0,1],x_{t},z_{r}\in\mu,y\in S\Rightarrow(x\gamma y\delta z)_{\min(t,r)}\in\mu.$
\end{theorem}

\begin{proof} $(B1)\Rightarrow(B2):$ Let $(B1)$ holds. Let $x,y,z\in S,\gamma,\delta\in\Gamma$ and $t,r\in(0,1]$ be such that $x_{t},z_{r}\in\mu.$ Then $\mu(x)\geq t$ and $\mu(z)\geq r.$ So, by $(S1)$ we have $\mu(x\gamma y\delta z)\geq\min\{\mu(x),\mu(z)\}\geq\min\{t,r\}.$ Hence $(x\gamma y\delta z)_{\min(t,r)}\in\mu.$

$(B2)\Rightarrow(B1):$ Let $(B2)$ holds. Let $x,y,z\in S,\gamma,\delta\in\Gamma.$ Since $x_{\mu(x)},y_{\mu(z)}\in\mu,$ so by $(S2)$, we have $(x\gamma y\delta z)_{\min\{\mu(x),\mu(z)\}}\in\mu$. Consequently, $\mu(x\gamma y\delta z)\geq\min\{\mu(x),\mu(z)\}.$ Hence the proof.
\end{proof}

\begin{remark}
Let $\mu$ be a fuzzy subset of a $\Gamma$-semigroup $S$ such that $\mu(x)\leq 0.5,\forall x\in S.$ Let $x\in S$ and $t\in (0,1]$ be such that $x_{t}\in\wedge q\mu.$ Then $\mu(x)\geq t$ and $\mu(x)+t>1.$ It follows that
$1<\mu(x)+t\leq\mu(x)+\mu(x)=2\mu(x).$ Hence $\mu(x)>0.5.$ This means that $\{x_{t}:x_{t}\in\wedge q\mu\}=\phi.$
\end{remark}

\begin{definition}
$(S3)$ A non-empty fuzzy subset $\mu$ of a $\Gamma$-semigroup $S$ is called an $(\alpha,\beta)$-fuzzy subsemigroup of $S,$ where $\alpha\neq\in\wedge q,$ if it satisfies $\forall x,y\in S,\forall\gamma\in\Gamma,\forall t,r\in(0,1],x_{t},y_{r}\alpha\mu\Rightarrow(x\gamma y)_{\min(t,r)}\beta\mu.$
\end{definition}

\begin{definition}
$(B3)$ An $(\alpha,\beta)$-fuzzy subsemigroup $\mu$ of a $\Gamma$-semigroup $S$ is called an $(\alpha,\beta)$-fuzzy bi-ideal of $S,$ where $\alpha\neq\in\wedge q,$ if it satisfies $\forall x,y,z\in S,\forall\gamma,\delta\in\Gamma,\forall t,r\in(0,1],x_{t},z_{r}\alpha\mu\Rightarrow(x\gamma y\delta z)_{\min(t,r)}\beta\mu.$
\end{definition}

\begin{example}
Let $S=\{a,b,c,d,e\}$ and $\Gamma =\{\gamma \}$, where $\gamma $ is defined on $S$ with the following cayley table:
\begin{equation*}
\begin{array}{l|lllll}
\gamma & a & b & c & d & e\\ \hline
a & a & d & a & d & d\\
b & a & b & a & d & d\\
c & a & d & c & d & e\\
d & a & d & a & d & d\\
e & a & d & c & d & e%
\end{array}
\end{equation*}

Then $S$ is a $\Gamma $-semigroup. We define fuzzy subset $\mu:S\rightarrow [0,1],$ by $\mu(a)=0.8,\mu(b)=0.7,\mu(c)=0.3,\mu(d)=0.5,\mu(e)=0.6.$

Clearly $\mu$ is an $(\in,\in \vee q)$-fuzzy subsemigroup and $(\in, \in \vee q)$-fuzzy bi-ideal of $S.$

Let $a_{0.78}\in\mu$ and $b_{0.66}\in\mu.$ But $(a\gamma b)_{\min(0.78,0.66)}=d_{0.66}\overline{\in}.$ Hence $\mu$ is not an $(\in, \in)$-fuzzy subsemigroup and $(\in,\in)$-fuzzy bi-ideal of $S.$ In a similar way we can show that $\mu$ is not $(q,\in),(\in,q),(q,\in\vee q),(q,\in\wedge q),(\in\vee q,\in\wedge q),(\in\vee q,\in),(\in,\in\wedge q),(q,q),(\in\vee q,q),(\in \vee q,\in\vee q)$-fuzzy subsemigroup and fuzzy bi-ideal
of $S.$
\end{example}

\begin{theorem}
Every $(\in \vee q,\in \vee q)$-fuzzy subsemigroup of a $\Gamma$-semigroup $S$ is $(\in, \in \vee q)$-fuzzy subsemigroup of $S.$
\end{theorem}

\begin{proof}
Let $x,y\in S,\gamma\in\Gamma$ and $t,r\in (0,1]$ be such that $x_{t},y_{r}\in\mu.$ Then $x_{t},y_{r}\in\vee q,$ which implies that $(x\gamma y)_{\min(t,r)}\in\vee q.$ Hence $\mu$ is an $(\in, \in  \vee q)$-fuzzy subsemigroup of $S.$
\end{proof}

In a similar fashion we can have the following theorem.

\begin{theorem}
Every $(\in \vee q,\in \vee q)$-fuzzy bi-ideal of a $\Gamma$-semigroup $S$ is $(\in, \in \vee q)$-fuzzy bi-ideal of $S.$
\end{theorem}

\begin{theorem}
Let $\mu$ be a non-zero $(\alpha,\beta)$-fuzzy subsemigroup of a $\Gamma$-semigroup $S.$ Then the set $\mu_{0}:=\{x\in S:\mu(x)>0\}$ is a subsemigroup of $S.$
\end{theorem}

\begin{proof}
Clearly $\mu_{0}$ is nonempty. Let $x,y\in\mu_{0},\gamma\in\Gamma.$ Then $\mu(x)>0$ and $\mu(y)>0.$ Let us assume that $\mu(x\gamma y)=0.$ If $\alpha\in\{\in,\in \vee q\}$ then $x_{\mu(x)}\alpha\mu$ and $y_{\mu(y)}\alpha\mu$ but $(x\gamma y)_{\min\{\mu(x),\mu(y)\}}\overline{\beta}\mu$ for every $\beta\in\{\in,q,\in\vee q,\in\wedge q\}$, a contradiction. If $\alpha=q$, it should be noted that $x_{1}q\mu$ and $y_{1}q\mu$ but $(x\gamma y)_{\min\{1,1\}}=(x\gamma y)_{1}\overline{\beta}\mu$ for every $\beta\in\{\in, q, \in \vee q,\in \wedge q\}$, a contradiction. Hence $\mu(x\gamma y)>0,$ consequently, $x\gamma y\in\mu_{0}.$
This completes the proof.
\end{proof}

In a similar way we can prove the following theorem.

\begin{theorem}
Let $\mu$ be a non-zero $(\alpha,\beta)$-fuzzy bi-ideal of a $\Gamma$-semigroup $S.$ Then the set $\mu_{0}:=\{x\in S:\mu(x)>0\}$ is a bi-ideal of $S.$
\end{theorem}

\begin{proposition}
$\cite{S5}$ Let $S$ be a $\Gamma$-semigroup and $A,B\subseteq S.$ Then $(1)$ $A\subseteq B$ if and only if $\mu_{A}\subseteq\mu_{B}.$ $(2)$ $\mu_{A}\cap\mu_{B}=\mu_{A\cap B}.$ ${3}$ $\mu_{A}\circ\mu_{B}=\mu_{A\Gamma B},$ where $\mu_{A},\mu_{B}$ denote the characteristic functions of $A$ and $B$ respectively.
\end{proposition}

\begin{definition}
Let $S$ be a $\Gamma$-semigroup. Then for any $a\in S$ and $\gamma\in\Gamma,$ we define $A_{a}:=\{(y,z)\in S\times S:a=y\gamma z$\}. For any two fuzzy subset $\mu_{1},\mu_{2}$ of $S,$ we define $(\mu _{1}\circ \mu _{2})(a)$ $:=\left\{
\begin{array}{l}
\underset{(y,z)\in A_{a}}{\sup }\min \{\mu _{1}(y),\mu _{2}(z)\}\text{ \ if }%
A_{a}\neq \phi  \\
\text{ \ \ \ \ \ \ \ \ \ \ \ \ }0\text{\ \ \ \ \ \ \ \ \ \ \ \ \
\ \ \ \ \ \ \ \ if }A_{a}=\phi
\end{array}%
\right. .$
\end{definition}

\begin{definition}
Let $S$ be a $\Gamma$-semigroup. For any two fuzzy subsets $\mu_{1}$ and $\mu_{2}$ of $S,$ we define the $0.5$-product of $\mu_{1}$ and $\mu_{2}$ by,\\ $(\mu _{1}\circ _{0.5}\mu _{2})(a)$
$:=\left\{
\begin{array}{l}
\underset{(y,z)\in A_{a}}{\sup }\min \{\mu _{1}(y),\mu
_{2}(z),0.5\}\text{ \
if }A_{a}\neq \phi  \\
\text{ \ \ \ \ \ \ \ \ \ \ \ \ }0\text{\ \ \ \ \ \ \ \ \ \ \ \ \
\ \ \ \ \ \ \ \ \ \ \ \ \ if }A_{a}=\phi
\end{array}%
\right. .$
\end{definition}

\begin{definition}
\cite{Yu1} For a $\Gamma$-semigroup $S$ the fuzzy subset $0($respectively $1)$ is defined by $0: S\rightarrow[0,1]:x\rightarrow 0(x)=0$, $\forall x\in\ S($respectively $1: S\rightarrow[0,1]:x\rightarrow 1(x)=1$ $\forall x\in\ S).$
\end{definition}

\begin{definition}
For any fuzzy subset $\mu$ of a $\Gamma$-semigroup $S$ and $t\in(0,1],$ we have $Q(\mu;t):=\{x\in S:x_{t}q \mu\}$ and $[\mu]_{t}:=\{x\in S:x_{t}\in\vee q\mu\}.$ It is clear that $[\mu]_{t}=U(\mu;t)\cup Q(\mu;t).$
\end{definition}

\begin{definition}
Let $\mu_{1},\mu_{2}$ be any two fuzzy subsets of a $\Gamma$-semigroup $S.$ Then $\forall x\in S,(\mu_{1}\cap_{0.5}\mu_{2})(x)$=min\{$\mu_{1}(x),\mu_{2}(x),0.5$\}.
\end{definition}

\begin{proposition}
Let S be a $\Gamma$-semigroup and $A,B\subseteq S$. Then $(1)$ $\mu_{A}\cap_{0.5}\mu_{B}=\mu_{A\cap B}\cap 0.5_{S}.$ $(2)$ $\mu_{A}\circ_{0.5}\mu_{B}=\mu_{A\Gamma B}\cap 0.5_{S}.$
\end{proposition}

\begin{proof}
$(1)$ Let $x\in S$. If $(\mu_{A}\cap_{0.5}\mu_{B})(x)=0\Rightarrow\min\{\mu_{A}(x),\mu_{B}(x),0.5\}=0$. Then $\mu_{A}(x)=0$ or $\mu_{B}(x)=0\Rightarrow x\notin A\cap B$. Then $\mu_{A\cap B}(x)=0\Rightarrow(\mu_{A\cap B}\cap 0.5_{S})(x)=0$. If $(\mu_{A}\cap_{0.5}\mu_{B})(x)\neq 0$, then $(\mu_{A}\cap_{0.5}\mu_{B})(x)=\min\{\mu_{A}(x),\mu_{B}(x),0.5\}=0.5$. Then $\mu_{A}(x)=1=\mu_{B}(x)\Rightarrow x\in A\cap B$ $\Rightarrow \mu_{A\cap B}(x)=1$. Therefore $(\mu_{A\cap B}\cap 0.5_{S})(x)=\min\{\mu_{A\cap B}(x),0.5\}=\min\{1,0.5\}=0.5$. Then $(\mu_{A}\cap_{0.5}\mu_{B})(x)=(\mu_{A\cap B}\cap 0.5_{S})(x),\forall x\in S$. Hence $\mu_{A}\cap_{0.5}\mu_{B}=\mu_{A\cap B}\cap 0.5_{S}$.

$(2)$ If $(\mu_{A\Gamma B}\cap 0.5_{S})(x)=0\Rightarrow \min\{\mu_{A\Gamma B}(x), 0.5\}=0$. Then $\mu_{A\Gamma B}(x)=0\Rightarrow x\notin A\Gamma B$. Then $x\neq a\alpha b$ $\forall a\in A,b\in B,\alpha\in \Gamma$. Then $(\mu_{A}\circ_{0.5}\mu_{B})(x)=0$. If $(\mu_{A\Gamma B}\cap 0.5_{S})(x)\neq 0$, then $(\mu_{A\Gamma B}\cap 0.5_{S})(x)=\min\{\mu_{A\Gamma B}(x),0.5\}=0.5$. Then $\mu_{A\Gamma B}(x)=1\Rightarrow x\in A\Gamma B$. Then $(\mu_{A}\circ_{0.5}\mu_{B})(x)=\underset{(u,v)\in A_{x}}{\sup}\min\{\mu_{A}(u),\mu_{B}(v),0.5\}=0.5$. Then $(\mu_{A}\circ_{0.5}\mu_{B})(x)=(\mu_{A\Gamma B}\cap 0.5_{S})(x)\forall x\in S$. Hence $\mu_{A}\circ_{0.5}\mu_{B}=\mu_{A\Gamma B}\cap 0.5_{S}$.
\end{proof}

\begin{proposition}
Let $\mu$ and $\nu$ be $(\in, \in \vee q)$-fuzzy subsemigroups of a $\Gamma$-semigroup $S.$ Then $\mu\circ_{0.5}\nu$ is an $(\in, \in \vee q)$-fuzzy bi-ideal of $S$ if any one of $\mu$ and $\nu$ be $(\in, \in \vee q)$-fuzzy bi-ideal of $S.$
\end{proposition}

\begin{proof}
Let $\mu$ be $(\in, \in \vee q)$-fuzzy bi-ideal of S and $a\in S.$ If $((\mu\circ_{0.5}\nu)\circ_{0.5}(\mu\circ_{0.5}\nu))(a)=0 \leq(\mu\circ_{0.5}\nu)(a)$. If $((\mu\circ_{0.5}\nu)\circ_{0.5}(\mu\circ_{0.5}\nu))(a)\neq0$, then $A_{a}\neq\phi$ and there exist $x,y,z,r,s,t\in S,\alpha,\beta,\gamma\in\Gamma$ such that $a=x\alpha y,x=z\beta r,y=s\gamma t.$ Then $a=(z\beta r)\alpha(s\gamma t)=(z\beta r\alpha s)\gamma t$. Consequently, $(z\beta r\alpha s,t)\in A_{a}.$ Then
\begin{align*}
((\mu\circ_{0.5}\nu)\circ_{0.5}(\mu\circ_{0.5}\nu))(a)&=\underset{(x,y)\in
A_{a}}{\sup}\min\{(\mu\circ_{0.5}\nu)(x),(\mu\circ_{0.5}\nu)(y),0.5\}\\
&=\underset{(x,y)\in A_{a}}{\sup}\min[\underset{(z,r)\in A_{x}}{\sup} \min\{\mu(z),\nu(r),0.5\},\\ &\underset{(s,t)\in
A_{y}}{\sup}\min\{\mu(s),\nu(t),0.5\},0.5]\\
&=\underset{(x,y)\in A_{a}}{\sup}\underset{(z,r)\in A_{x}}{\sup}\underset{(s,t)\in
A_{y}}{\sup}\min\{\mu(z),\mu(s),\nu(r),\nu(t),0.5\}\\
&\leq\underset{(x,y)\in A_{a}}{\sup}\underset{(z,r)\in A_{x}}{\sup}\underset{(s,t)\in
A_{y}}{\sup}\min\{\mu(z),\mu(s),\nu(t),0.5\}
\end{align*}
\begin{align*}
&\leq\underset{(z\beta r\alpha s,t)\in A_{a}}{\sup}\min\{\mu(z),\mu(s),\nu(t),0.5\}\\
&\leq\underset{(z\beta r\alpha s,t)\in A_{a}}{\sup}\min\{\mu(z\beta r\alpha
s),\nu(t),0.5\}\\
&\leq\underset{(x,y)\in A_{a}}{\sup}\min\{\mu(x),\nu(y),0.5\}\\
&=(\mu\circ_{0.5}\nu)(a).
\end{align*}
Then $(\mu\circ_{0.5}\nu)\circ_{0.5}(\mu\circ_{0.5}\nu)\subseteq\mu\circ_{0.5}\nu$. Hence $\mu\circ_{0.5}\nu$ is an $(\in, \in \vee q)$-fuzzy subsemigroup of $S.$

Now if $((\mu\circ_{0.5}\nu)\circ_{0.5}1\circ_{0.5}(\mu\circ_{0.5}\nu))(a)=0\leq(\mu\circ_{0.5}\nu)(a)$. If $((\mu\circ_{0.5}\nu)\circ_{0.5}1\circ_{0.5}(\mu\circ_{0.5}\nu))(a)\neq0$, $A_{a}\neq\phi$. Then there exists $x,y,w,z,b,c,d,e\in S,\alpha,\beta,\gamma,\delta\in\Gamma$ such that $a=x\alpha y,x=w\beta z,y=b\gamma c, c=d\delta e$. Then $a=(w\beta z)\alpha(b\gamma c)=w\beta z\alpha b\gamma(d\delta e)=(w\beta z\alpha b\gamma d)\delta e$. Consequently, $(w\beta z\alpha b\gamma d,e)\in A_{a}$. Then
\begin{align*}
((\mu\circ_{0.5}\nu)\circ_{0.5}1\circ_{0.5}(\mu\circ_{0.5}\nu))(a)&=\underset{(x,y)\in A_{a}}{\sup}\min[(\mu\circ_{0.5}\nu)(x),(1\circ_{0.5}(\mu\circ_{0.5}\nu))(y),0.5]\\
&=\underset{(x,y)\in A_{a}}{\sup}\min[\underset{(w,z)\in A_{x}}{\sup}\min\{\mu(w),\nu(z),0.5\},\\
&\underset{(b,c)\in A_{y}}{\sup}\min\{1(b),(\mu\circ_{0.5}\nu)(c),0.5\},0.5]\\
&=\underset{(x,y)\in A_{a}}{\sup}\min[\underset{(w,z)\in A_{x}}{\sup}\min\{\mu(w),\nu(z),0.5\},\\
&\underset{(b,c)\in A_{y}}{\sup}\min\{\underset{(d,e)\in A_{c}}{\sup}\min\{\mu(d),\nu(e),0.5\},0.5\},0.5]\\
&=\underset{(x,y)\in A_{a}}{\sup}\underset{(w,z)\in A_{x}}{\sup}\underset{(b,c)\in A_{y}}{\sup}
\underset{(d,e)\in A_{c}}{\sup}\min\{\mu(w),\mu(d),\nu(z),\\
&\nu(e),0.5\}\\
&\leq\underset{(x,y)\in A_{a}}{\sup}\underset{(w,z)\in A_{x}}{\sup}\underset{(b,c)\in A_{y}}{\sup}
\underset{(d,e)\in A_{c}}{\sup}\min\{\mu(w),\mu(d),\nu(e),\\
&0.5\}\\
&\leq\underset{(w\beta z\alpha b\gamma d,e)\in A_{a}}{\sup}\min\{\mu(w\beta z\alpha b\gamma d),\nu(e),0.5\}\\
&\leq\underset{(x,y)\in A_{a}}{\sup}\min\{\mu(x),\nu(y),0.5\}\\
&=(\mu\circ_{0.5}\nu)(a).
\end{align*}
Then
$(\mu\circ_{0.5}\nu)\circ_{0.5}1\circ_{0.5}(\mu\circ_{0.5}\nu)\subseteq\mu\circ_{0.5}\nu$. Hence $\mu\circ_{0.5}\nu$ is an $(\in, \in \vee q)$-fuzzy bi-ideal of $S.$
\end{proof}

\begin{proposition}
If $\mu_{1},\mu_{2}$ be any two $(\in,\in \vee q)$-fuzzy subsemigroups$($fuzzy bi-ideals$)$ of a $\Gamma$-semigroup $S,$ then $(\mu_{1}\cap_{0.5}\mu_{2})$ is an $(\in, \in \vee q)$-fuzzy subsemigroup$($resp. fuzzy bi-ideal$)$ of $S.$
\end{proposition}

\begin{proof}
Let $\mu_{1},\mu_{2}$ be any two $(\in, \in \vee q)$-fuzzy subsemigroups of $S$ and $x,y\in S,\gamma\in\Gamma.$ Then
\begin{align*}
(\mu_{1}\cap_{0.5}\mu_{2})(x\gamma y)&=\min\{\mu_{1}(x\gamma y),\mu_{2}(x\gamma y),0.5\}\\
&\geq\min\{\min\{\mu_{1}(x),\mu_{1}(y),0.5\},\min\{\mu_{2}(x),\mu_{2}(y),0.5\},0.5\}\\
&=\min\{\min\{\mu_{1}(x),\mu_{1}(y),\mu_{2}(x),\mu_{2}(y)\},0.5\}\\
&=\min\{\min\{\mu_{1}(x),\mu_{2}(x),0.5\},\min\{\mu_{1}(y),\mu_{2}(y),0.5\},0.5\}\\
&=\min\{(\mu_{1}\cap_{0.5}\mu_{2})(x),(\mu_{1}\cap_{0.5}\mu_{2})(y),0.5\}.
\end{align*}
Hence $(\mu_{1}\cap_{0.5}\mu_{2})$ is an $(\in, \in \vee q)$-fuzzy subsemigroup of $S.$ Similarly we can prove the other case also.
\end{proof}

\begin{definition}
Let $S$ be a $\Gamma$-semigroup and $\mu$ be a non-empty fuzzy subset of $S.$ Then $\mu$ is called $(\in, \in \vee q)$-fuzzy left ideal$($fuzzy right ideal$)$ of $S$ if $\forall x,y\in S,\forall\gamma\in\Gamma,\mu(x\gamma y)\geq\min\{\mu(y),0.5\}($resp. $\mu(x\gamma y)\geq\min\{\mu(x),0.5\})$.

A non empty fuzzy subset $\mu$ of a $\Gamma$-semigroup $S$ is called $(\in, \in \vee q)$-fuzzy ideal of $S$ if it is both $(\in, \in \vee q)$-fuzzy left ideal and $(\in, \in \vee q)$-fuzzy right ideal of $S.$
\end{definition}

\begin{proposition}
Let $S$ be a $\Gamma$-semigroup. Then every one sided $(\in,\in\vee q)$-fuzzy ideal of $S$ is an $(\in, \in \vee q)$-fuzzy bi-ideal of $S.$
\end{proposition}

\begin{proof}
Let $\mu$ be an $(\in, \in \vee q)$-fuzzy right ideal of $S$ and let $x,y,z\in S,\alpha,\beta\in\Gamma.$ Then $\mu(x\alpha y\beta z)=\mu(x\alpha (y\beta z))\geq\min\{\mu(x),0.5\}\geq\min\{\mu(x),\mu(z),0.5\}.$ Hence
$\mu$ is an $(\in,\in \vee q)$-fuzzy bi-ideal of $S.$ Similarly we can prove the proposition by taking $\mu$ as $(\in, \in  \vee q)$-fuzzy left ideal of $S.$
\end{proof}

\begin{proposition}
Let $S$ be a regular left duo$($right duo, duo$)$ $\Gamma$-semigroup. Then every $(\in, \in\vee q)$-fuzzy bi-ideal of $S$ is an $(\in, \in \vee q)$-fuzzy right ideal$($resp. fuzzy left ideal, fuzzy ideal$)$ of $S.$
\end{proposition}

\begin{proof}
Let $\mu$ be an $(\in, \in \vee q)$-fuzzy bi-ideal of $S$ and $x,y\in S,\gamma\in\Gamma.$ Then $x\gamma y\in S.$ Since $S$ is regular and left duo in view of the fact that $S\Gamma x$ is a left ideal we obtain, $x\gamma y\in(x\Gamma S\Gamma x)\Gamma S\subseteq x\Gamma S\Gamma x.$ This implies that there exist elements $z\in S,\alpha,\beta\in\Gamma,$ such that $x\gamma y=x\alpha z\beta x.$ Then $\mu(x\gamma y)=\mu(x\alpha z\beta
x)\geq\min\{\mu(x),\mu(x),0.5\}=\min\{\mu(x),0.5\}.$ Hence $\mu$ is an $(\in, \in \vee q)$-fuzzy right ideal of $S.$ Similarly we can prove the other cases also.
\end{proof}

In view of above two proposition we have the following theorem.

\begin{theorem}
In a regular left duo$($right duo,duo$)$ $\Gamma$-semigroup the following conditions are equivalent: $(1)$ $\mu$ is an $(\in, \in \vee q)$-fuzzy right ideal$($resp. fuzzy left ideal, fuzzy ideal$)$ of $S,$ $(2)$ $\mu$ is an $(\in,\in \vee q)$-fuzzy bi-ideal of $S.$
\end{theorem}

We can prove the following theorem easily.

\begin{theorem}
Let $S$ be a $\Gamma$-semigroup. Then a non-empty subset $A$ of $S$ is a subsemigroup$($bi-ideal$)$ of $S$ if and only if the characteristic function $\mu_{A}$ of $A$ is an $(\in,\in \vee q)$-fuzzy subsemigroup$($resp. fuzzy bi-ideal$)$ of $S.$
\end{theorem}

\begin{theorem}
Let $S$ be a $\Gamma$-semigroup and $\mu$ be a non-empty fuzzy subset of $S.$ Then $\mu$ is an $(\in, \in \vee q)$-fuzzy subsemigroup of $S$ if and only if $[\mu]_{t}$ is a subsemigroup of $S.$
\end{theorem}

\begin{proof}
Let $\mu$ be an $(\in, \in \vee q)$-fuzzy subsemigroup of $S.$ Let $x,y\in[\mu]_{t}$,$\gamma\in\Gamma$ and $t\in(0,1].$ Then $x_{t},y_{t}\in\vee q\mu,$ which implies, $\mu(x)\geq t$ or $\mu(x)+t>1$ and $\mu(y)\geq t$ or $\mu(y)+t>1.$ Since $\mu$ is an $(\in, \in \vee q)$-fuzzy subsemigroup of $S,$ we have $\mu(x\gamma y)\geq\min\{\mu(x),\mu(y),0.5\}.$\\

Case-$(1)$ Let $\mu(x)\geq t$ and $\mu(y)\geq t.$ Then $\mu(x\gamma y)\geq\min\{t,0.5\}.$ If $t>0.5,$ then $\mu(x\gamma y)\geq 0.5$ and consequently, $(x\gamma y)_{t}q\mu.$ If $t\leq 0.5,$ then $\mu(x\gamma y)\geq t$ and so $(x\gamma y)_{t}\in\mu.$ Hence $(x\gamma y)_{t}\in\vee q\mu,$ implies that $x\gamma y\in[\mu]_{t}.$\\

Case-$(2)$ Let $\mu(x)\geq t$ and $\mu(y)+t> 1.$ Then $\mu(x\gamma y)\geq\min\{t,1-t,0.5\}.$ If $t>0.5,$ then
$\mu(x\gamma y)> 1-t,i.e.,\mu(x\gamma y)+t>1$ and consequently, $(x\gamma y)_{t}q\mu.$ If $t\leq0.5,$ then $\mu(x\gamma y)\geq t$ and so $(x\gamma y)_{t}\in\mu.$ Hence $(x\gamma y)_{t}\in\vee q\mu,$ implies that $x\gamma y\in[\mu]_{t}.$\\

Case-$(3)$ Let $\mu(x)+t> 1$ and $\mu(y)\geq t.$ Then $\mu(x\gamma y)\geq\min\{1-t,t,0.5\}.$ If $t>0.5,$ then
$\mu(x\gamma y)> 1-t,i.e.,\mu(x\gamma y)+t>1$ and consequently, $(x\gamma y)_{t}q\mu.$ If $t\leq0.5,$ then $\mu(x\gamma y)\geq t$ and so $(x\gamma y)_{t}\in\mu.$ Hence $(x\gamma y)_{t}\in\vee q\mu,$ implies that $x\gamma y\in[\mu]_{t}.$\\

Case-$(4)$ Let $\mu(x)+t> 1$ and $\mu(y)+t> 1.$ Then $\mu(x\gamma y)\geq\min\{1-t,1-t,0.5\}.$ If $t>0.5,$ then $\mu(x\gamma y)>1-t,i.e.,\mu(x\gamma y)+t>1$ and consequently, $(x\gamma y)_{t}q\mu.$ If $t\leq0.5,$ then $\mu(x\gamma y)\geq 0.5\geq t$ and so $(x\gamma y)_{t}\in\mu.$ Hence $(x\gamma y)_{t}\in\vee q\mu,$ implies that $x\gamma y\in[\mu]_{t}.$ Thus in any case, we have $x\gamma y\in[\mu]_{t}.$ Hence $[\mu]_{t}$ is a subsemigroup
of $S.$\\

Conversely, Let $[\mu]_{t}$ is a subsemigroup of $S$ and let $x,y\in Supp(\mu),\gamma\in\Gamma$ be such that $\mu(x\gamma y)<t<\min\{\mu(x),\mu(y),0.5\}$ for some $t\in (0,0.5].$ Then $x,y\in U(\mu;t)\subseteq [\mu]_{t},$ which implies $x\gamma y\in[\mu]_{t}($since $[\mu]_{t}$ is a subsemigroup of $S$). Consequently, $\mu(x\gamma y)\geq t$ or $\mu(x\gamma y)+t>1,$ a contradiction. Thus $\mu(x\gamma y)\geq\min\{\mu(x),\mu(y),0.5\}\forall x,y\in S,\forall\gamma\in\Gamma.$ Hence $\mu$ is an $(\in,\in \vee q)$-fuzzy subsemigroup of $S.$
\end{proof}

In a similar way we can have the following theorem.

\begin{theorem}
Let $S$ be a $\Gamma$-semigroup and $\mu$ be a non-empty fuzzy subset of $S.$ Then $\mu$ is an $(\in, \in \vee q)$-fuzzy bi-ideal of $S$ if and only if $[\mu]_{t}$ is a bi-ideal of $S.$
\end{theorem}

\begin{theorem}
A fuzzy subset $\mu$ of a $\Gamma$-semigroup $S$ is an $(\in, \in\vee q)$-fuzzy subsemigroup of $S$ if and only if $\mu\circ_{0.5}\mu\subseteq\mu.$
\end{theorem}

\begin{proof}
Let $\mu$ be an $(\in, \in \vee q)$-fuzzy subsemigroup of $S.$ Let $a\in S$. If $A_{a}=\phi,$ then $(\mu\circ_{0.5}\mu)(a)=0\leq\mu(a).$ If $A_{a}\neq\phi$, then
\begin{align*}
(\mu\circ_{0.5}\mu)(a)&=\underset{(y,z)\in A_{a}}{\sup}[\min\{\mu(y),\mu(z),0.5\}]\\
&\leq\underset{(y,z)\in A_{a}}{\sup}\mu(y\gamma z)(\text{where } a=y\gamma z)\\
&=\underset{(y,z)\in A_{a}}{\sup}\mu(a)=\mu(a).
\end{align*}
Hence $\mu\circ_{0.5}\mu\subseteq\mu.$

Conversely, let $\mu\circ_{0.5}\mu\subseteq\mu.$ Then for $y,z\in S,\gamma\in\Gamma$,
\begin{align*}
\mu(y\gamma z)&\geq(\mu\circ_{0.5}\mu)(y\gamma z)\\
&=\underset{(b,c)\in A_{y\gamma z}}{\sup}\min\{\mu(y),\mu(z),0.5\}\\
&\geq\min\{\mu(y),\mu(z),0.5\}.
\end{align*}
Hence $\mu$ is an $(\in, \in \vee q)$-fuzzy subsemigroup of $S.$
\end{proof}

\begin{theorem}
In a $\Gamma$-semigroup $S$ the following are equivalent: $(1)$ $\mu$ is an $(\in, \in \vee q)$-fuzzy bi-ideal of $S,$ $(2)$ $\mu\circ_{0.5}\mu\subseteq\mu$ and $\mu\circ_{0.5}1\circ_{0.5}\mu\subseteq\mu.$
\end{theorem}

\begin{proof}
Let us assume that $(1)$ hold. Since $\mu$ is an $(\in, \in \vee q)$-fuzzy bi-ideal of $S,$ then $\mu$ is an $(\in,\in \vee q)$-fuzzy subsemigroup of $S.$ So by, Theorem $4.25,$ $\mu\circ_{0.5}\mu\subseteq\mu.$

Let $a\in S.$ If $A_{a}=\phi.$ Then $(\mu\circ_{0.5}1\circ_{0.5}\mu)(a)=0\leq\mu(a).$

If $A_{a}\neq\phi,$ then there exist $x,y,p,q\in S$ and $\beta,\gamma\in\Gamma$ such that $a=x\gamma y$ and $x=p\beta q.$ Since $\mu$ is an $(\in,\in\vee q)$-fuzzy bi-ideal of $S,$ we have $\mu(a)=\mu(x\gamma y)=\mu(p\beta q\gamma y)\geq\min\{\mu(p),\mu(y),0.5\}.$ Then
\begin{align*}
(\mu\circ_{0.5}1\circ_{0.5}\mu)(a)&=\underset{(x,y)\in A_{a}}{\sup}[\min\{(\mu\circ_{0.5}1)(x),\mu(y)\}]\\
&=\underset{(x,y)\in A_{a}}{\sup}[\min\{\underset{(p,q)\in A_{x}}{\sup}\{\min\{\mu(p),1(q),0.5\}\},\mu(y)\}]\\
&=\underset{(x,y)\in A_{a}}{\sup}[\min\{\underset{(p,q)\in
A_{x}}{\sup}\{\min\{\mu(p),1,0.5\}\},\mu(y)\}]\\
&=\underset{(x,y)\in A_{a}}{\sup}\underset{(p,q)\in A_{x}}{\sup}\min\{\mu(p),\mu(y),0.5\}\\
&\leq\underset{(x,y)\in A_{a}}{\sup}\{\mu(p\beta q\gamma y)\}\\
&=\underset{(x,y)\in A_{a}}{\sup}\mu(x\gamma y)=\mu(a).
\end{align*}
Hence $\mu\circ_{0.5}1\circ_{0.5}\mu\subseteq\mu.$

Conversely, let us suppose that $(2)$ holds. Since $\mu\circ_{0.5}\mu\subseteq\mu,$ so $\mu$ is an $(\in,\in \vee q)$-fuzzy subsemigroup of $S.$ Let $x,y,z\in S,\beta,\gamma\in\Gamma$ and $a=x\beta y\gamma z=p\gamma z($where
$p=x\beta y).$ Since $\mu\circ_{0.5}1\circ_{0.5}\mu\subseteq\mu,$ we have
\begin{align*}
\mu(x\beta y\gamma z)&=\mu(a)\geq(\mu\circ_{0.5}1\circ_{0.5}\mu)(a)\\
&=\underset{(p,z)\in A_{a}}{\sup}[\min\{(\mu\circ_{0.5}1)(p),\mu(z),0.5\}]\\
&\geq\min\{(\mu\circ_{0.5}1)(p),\mu(z),0.5\}\\
&=\min[\underset{(x,y)\in A_{a}}{\sup}\{\min\{\mu(x),1(y),0.5\},\mu(z),0.5]\\
&=\min[\underset{(x,y)\in A_{a}}{\sup}\{\min\{\mu(x),1,0.5\},\mu(z),0.5]\\
&\geq\min[\min\{\mu(x),0.5\},\mu(z),0.5]\\
&=\min\{\mu(x),\mu(z),0.5\}.
\end{align*}
Hence $\mu$ is an $(\in, \in \vee q)$-fuzzy bi-ideal of $S.$
\end{proof}

In the following example we show that in Theorem $4.26$ equality does not hold generally in condition $(2).$

\begin{example}
Let $S=\{a,b,c\}$ and $\Gamma=\{\gamma\}$, where $\gamma$ is defined on S with the following cayley table:
\begin{equation*}
\begin{array}{l|lll}
\gamma & a & b &c\\ \hline
a & a & a & a \\
b & b & b & b \\
c & c & c & c \\
\end{array}
\end{equation*}

Then $S$ is a $\Gamma $-semigroup. We define fuzzy subset $\mu:S\rightarrow [0,1],$ by $\mu(a)=0.8,\mu(b)=0.7,\mu(c)=0.6.$ Clearly $\mu$ is an $(\in,\in \vee q)$-fuzzy subsemigroup and $(\in, \in \vee q)$-fuzzy bi-ideal of $S.$

Now $(\mu\circ_{0.5}\mu)(a)$=$\underset{(x,y)\in A_{a}}{\sup}\min\{\mu(x),\mu(y),0.5\}=0.5 <0.8=\mu(a)$. Hence $\mu\neq\mu\circ_{0.5}\mu$. Also $(\mu\circ_{0.5}1\circ_{0.5}\mu)(a)=0.5<0.8=\mu(a)$. Hence $\mu\circ_{0.5}1\circ_{0.5}\mu\neq\mu$.
\end{example}

\begin{theorem}
A $\Gamma$-semigroup $S$ is regular if and only if for every $(\in,\in\vee q)$-fuzzy bi-ideal $\mu$ of $S,$ $\mu\circ_{0.5}1\circ_{0.5}\mu=\mu\cap 0.5_{S}$.
\end{theorem}

\begin{proof}
Let $S$ be a regular $\Gamma$-semigroup and $\mu$ be an $(\in,\in\vee q)$-fuzzy bi-ideal of $S.$ By Theorem $4.26,$ $\mu\circ_{0.5}1\circ_{0.5}\mu\subseteq\mu$. Also for $a\in S$,
\begin{align*}
(\mu\circ_{0.5}1\circ_{0.5}\mu)(a)&=\underset{(u,v)\in A_{a}}{\sup}\min\{\mu(u),(1\circ_{0.5}\mu)(v),0.5\}\leq 0.5=(0.5_{S})(a).
\end{align*}
Then
$\mu\circ_{0.5}1\circ_{0.5}\mu\subseteq 0.5_{S}$. Hence $\mu\circ_{0.5}1\circ_{0.5}\mu\subseteq\mu\cap 0.5_{S}$. Since $S$ is regular for $a\in S$ there exists $x\in S,\alpha,\beta\in\Gamma$ such that $a=a\alpha x\beta a$. Then
\begin{align*}
(\mu\circ_{0.5}1\circ_{0.5}\mu)(a)&=\underset{(u,v)\in A_{a}}{\sup}\min\{\mu(u),(1\circ_{0.5}\mu)(v),0.5\}\\
&\geq\min\{\mu(a),(1\circ_{0.5}\mu)(x\beta a),0.5\}\\
&\geq\min\{\mu(a),\min\{1(x),\mu(a),0.5\},0.5\}\\
&=\min\{\mu(a),0.5\}=(\mu\cap 0.5_{S})(a).
\end{align*}
So $\mu\circ_{0.5}1\circ_{0.5}\mu\supseteq \mu\cap 0.5_{S}$. Hence $\mu\circ_{0.5}1\circ_{0.5}\mu=\mu\cap 0.5_{S}$.

Conversely, let for every $(\in,\in\vee q)$-fuzzy bi-ideal $\mu$ of $S,$ $\mu\circ_{0.5}1\circ_{0.5}\mu=\mu\cap 0.5_{S}$. Let $B$ be a bi-ideal of $S.$ Then $\mu_{B}$ is an $(\in,\in\vee q)$-fuzzy bi-ideal of $S.$ Then $\mu_{B}\circ_{0.5}1\circ_{0.5}\mu_{B}=\mu_{B}\cap 0.5_{S}$. Then $\mu\circ_{0.5}1\circ_{0.5}\mu=\mu\cap 0.5_{S}$. Then by Proposition $4.15,$ $\mu_{B\Gamma S\Gamma B}\cap 0.5_{S}=\mu_{B}\cap 0.5_{S}$. If $b\in B$, then $\mu(b)=1$. Then $(\mu_{B}\cap 0.5_{S})(b)=\min\{\mu_{B}(b),0.5\}=\min\{1,0.5\}=0.5.$ Then $(\mu_{B\Gamma S\Gamma B}\cap 0.5_{S})(b)=0.5 \Rightarrow\min\{\mu_{B\Gamma S\Gamma B}(b),0.5\}=0.5 \Rightarrow\mu_{B\Gamma S\Gamma B}(b)=1\Rightarrow b\in B\Gamma S\Gamma B$. Then $B\subseteq B\Gamma S\Gamma B$. Also $B\Gamma S\Gamma B\subseteq B$. So $B=B\Gamma S\Gamma B$. Hence $S$ is regular.
\end{proof}

\begin{theorem}
Let $S$ be a $\Gamma$-semigroup. Then the following are equivalent: $(1)$ $S$ is both regular and intra-regular, $(2)$ $\mu\circ_{0.5}\mu=\mu\cap 0.5_{S}$, for every $(\in,\in\vee q)$-fuzzy bi-ideal $\mu$ of $S,$ $(3)$ $\mu\cap_{0.5}\nu=(\mu\circ_{0.5}\nu)\cap_{0.5}(\nu\circ_{0.5}\mu)$ for all $(\in,\in\vee q)$-fuzzy bi-ideals $\mu$ and $\nu$ of $S.$
\end{theorem}

\begin{proof}
$(1)\Rightarrow(2):$ Let $S$ be both regular and intra-regular. Let $\mu$ be an $(\in,\in\vee q)$-fuzzy bi-ideal of $S.$ Then by Theorem $4.26,$ $\mu\circ_{0.5}\mu\subseteq\mu$. Also $\mu\circ_{0.5}\mu\subseteq 0.5_{S}$. Then $\mu\circ_{0.5}\mu\subseteq\mu\cap 0.5_{S}$. Let $a\in S$. Since $S$ is regular and intra-regular, there exists $x,y,z\in S$ and $\alpha,\beta,\gamma,\delta,\eta\in\Gamma$ such that $a=a\alpha x\beta a$ and $a=y\gamma a\delta a\eta z$. Hence
\begin{align*}
a&=a\alpha x\beta a=a\alpha x\beta a\alpha x\beta a=a\alpha x\beta(y\gamma a\delta a\eta z)\alpha x\beta a=(a\alpha x\beta y\gamma a)\delta(a\eta z\alpha x\beta a).
\end{align*}
Then
\begin{align*}
(\mu\circ_{0.5}\mu)(a)&=\underset{(u,v)\in A_{a}}{\sup}\min\{\mu(u),\mu(v),0.5\}\\
&\geq\min\{\mu(a\alpha x\beta y\gamma a),\mu(a\eta z\alpha x\beta a),0.5\}\\
&\geq\min\{\min\{\mu(a),\mu(a),0.5\},\min\{\mu(a),\mu(a),0.5\},0.5\}\\
&=\min\{\mu(a),0.5\}=(\mu\cap 0.5_{S})(a).
\end{align*}
Then $\mu\circ_{0.5}\mu\supseteq\mu\cap 0.5_{S}$. Hence $\mu\circ_{0.5}\mu=\mu\cap 0.5_{S}$.\\

$(2)\Rightarrow(3):$ Let $\mu$ and $\nu$ be two $(\in,\in\vee q)$-fuzzy bi-ideals of $S.$ Then by Proposition $4.17,$ $\mu\cap_{0.5}\nu$ is also an $(\in,\in\vee q)$-fuzzy bi-ideals of $S.$ Then $(\mu\cap_{0.5}\nu)\circ_{0.5}(\mu\cap_{0.5}\nu)=(\mu\cap_{0.5}\nu)\cap 0.5_{S}=\mu\cap_{0.5}\nu$. Then $\mu\cap_{0.5}\nu=(\mu\cap_{0.5}\nu)\circ_{0.5}(\mu\cap_{0.5}\nu)\subseteq\mu\circ_{0.5}\nu$. Similarly $\mu\cap_{0.5}\nu\subseteq\nu\circ_{0.5}\mu$. Hence $\mu\cap_{0.5}\nu\subseteq(\mu\circ_{0.5}\nu)\cap_{0.5}(\nu\circ_{0.5}\mu)$. Also by Proposition $4.16,$ $\mu\circ_{0.5}\nu$ and $\nu\circ_{0.5}\mu$ are $(\in,\in\vee q)$-fuzzy bi-ideals of $S.$ Then by Proposition $4.17,$ $(\mu\circ_{0.5}\nu)\cap_{0.5}(\nu\circ_{0.5}\mu)$ is an $(\in,\in\vee q)$-fuzzy bi-ideal of $S.$ Then
\begin{align*}
(\mu\circ_{0.5}\nu)\cap_{0.5}(\nu\circ_{0.5}\mu)&=((\mu\circ_{0.5}\nu)\cap_{0.5}(\nu\circ_{0.5}\mu))\cap 0.5_{S}\\
&=((\mu\circ_{0.5}\nu)\cap_{0.5}(\nu\circ_{0.5}\mu))\circ_{0.5}((\mu\circ_{0.5}\nu)\cap_{0.5}
(\nu\circ_{0.5}\mu))\\
&\subseteq(\mu\circ_{0.5}\nu)\circ_{0.5}(\nu\circ_{0.5}\mu)\\
&=\mu\circ_{0.5}(\nu\circ_{0.5}\nu)\circ_{0.5}\mu\\
&\subseteq\mu\circ_{0.5}1\circ_{0.5}\mu\\
&\subseteq\mu\cap 0.5_{S}.
\end{align*}
Similarly $(\mu\circ_{0.5}\nu)\cap_{0.5}(\nu\circ_{0.5}\mu)\subseteq\nu\cap 0.5_{S}$. Then $(\mu\circ_{0.5}\nu)\cap_{0.5}(\nu\circ_{0.5}\mu)$
$\subseteq(\mu\cap 0.5_{S})\cap(\nu\cap 0.5_{S})$=$\mu\cap_{0.5}\nu$. Hence $\mu\cap_{0.5}\nu=(\mu\circ_{0.5}\nu)\cap_{0.5}(\nu\circ_{0.5}\mu).$\\

$(3)\Rightarrow(1):$ Let $(3)$ holds. To prove that $S$ is regular and intra-regular we have to prove that $P\cap Q=P\Gamma Q\cap Q\Gamma P$ for every bi-ideal $P$ and $Q$ of $S.$ Let $b\in P\cap Q$. By Theorem $4.22,$ $\chi_{P}$ and $\chi_{Q}$ are $(\in,\in\vee q)$-fuzzy bi-ideals of $S.$ Now $\chi_{P}(b)=1$ and $\chi_{Q}(b)=1.$ By $(3),$ $\chi_{P}\cap_{0.5}\chi_{Q}=(\chi_{P}\circ_{0.5}\chi_{Q})\cap_{0.5}(\chi_{Q}\circ_{0.5}\chi_{P}).$ Now $(\chi_{P}\cap_{0.5}\chi_{Q})(b)=\min\{\chi_{P}(b),\chi_{Q}(b),0.5\}=0.5.$ Then $((\chi_{P}\circ_{0.5}\chi_{Q})\cap_{0.5}(\chi_{Q}\circ_{0.5}\chi_{P}))(b)=0.5.$ By Proposition $4.15,$ $(\chi_{P\Gamma Q\cap Q\Gamma P}\cap 0.5_{S})(b)=0.5.$ Then $\chi_{P\Gamma Q\cap Q\Gamma P}(b)=1 \Rightarrow b\in P\Gamma Q\cap Q\Gamma P.$ Then $P\cap Q\subseteq P\Gamma Q\cap Q\Gamma P.$ If $b\in P\Gamma Q\cap Q\Gamma P.$ Then by Proposition $4.15,$ $\chi_{P\Gamma Q\cap Q\Gamma P}(b)=1\Rightarrow(\chi_{P\Gamma Q\cap Q\Gamma P}\cap 0.5_{S})(b)=0.5\Rightarrow (\chi_{P\Gamma Q}\cap_{0.5}\chi_{Q\Gamma P})(b)=0.5\Rightarrow((\chi_{P}\circ_{0.5}\chi_{Q})\cap_{0.5}(\chi_{Q}\circ_{0.5}\chi_{P}))(b)=0.5\Rightarrow (\chi_{P}\cap_{0.5}\chi_{Q})(b)=0.5\Rightarrow(\chi_{P\cap Q}\cap 0.5_{S})(b)=0.5\Rightarrow\chi_{P\cap Q}(b)=1.$ Then $b\in P\cap Q.$ So $P\cap Q\supseteq P\Gamma Q\cap Q\Gamma P.$ Hence $P\cap Q=P\Gamma Q\cap Q\Gamma P$. Consequently, $S$ is both regular and intra-regular.
\end{proof}



\begin{thebibliography}{9}
\bibitem{Bh1} S.K. Bhakat and P. Das; {\it On the definition of a fuzzy subgroup,} Fuzzy Sets Syst., 51 (1992) 235-241.\vspace{-.1 in}

\bibitem{Bh2} S.K. Bhakat and P. Das; {\it $(\in ,\in \vee q)$-fuzzy subgroup,} Fuzzy Sets Syst., 80 (1996) 359-368.\vspace{-.1 in}

\bibitem{Ch1} S. Chattopadhyay; {\it Right inverse $\Gamma$-semigroup.} Bull. Cal. Math. Soc., 93 (2001) 435-442.\vspace{-.1 in}

\bibitem{Ch2} S. Chattopadhyay; {\it  Right orthodox $\Gamma$-semigroup,} Southeast Asian Bull. Math., 29 (2005) 23-30. \vspace{-.1 in}

\bibitem{Chin} R. Chinram,; {\it On quasi-$\Gamma$-ideals in $\Gamma$-semigroups,} Science Asia, 32 (2006) 351-353.\vspace{-.1 in}

\bibitem{D1} T.K. Dutta and N.C. Adhikari; {\it On $\Gamma$-semigroup with the right and left unities,} Soochow J. Math., 19 (4) (1993) 461-474.\vspace{-.1 in}

\bibitem{D3} T.K. Dutta and N.C. Adhikari; {\it On prime radical of $\Gamma$-semigroup,} Bull. Cal. Math. Soc., 86 (5) (1994) 437-444.\vspace{-.1 in}

\bibitem{D2} T.K. Dutta, S.K. Sardar and S.K. Majumder; {\it Fuzzy ideal extensions of $\Gamma $-semigroups,} International Mathematical Forum, 4(41) (2009) 2093-2100.\vspace{-.1 in}

\bibitem{D4} T.K. Dutta, S.K. Sardar and S.K. Majumder; {\it Fuzzy ideal extensions of $\Gamma $-semigroups via its operator semigroups,} Int. J. Contemp. Math. Sciences, 4(30)(2009) 1455-1463.\vspace{-.1 in}

\bibitem{H1} K. Hila; {\it On regular, semiprime and quasi-reflexive $\Gamma$-semigroup and minimal quasi-ideals,} Lobachevski J. Math., 29 (2008) 141-152.\vspace{-.1 in}

\bibitem{H2} K. Hila; {\it On some classes of le-$\Gamma$-semigroup and minimal quasi-ideals,} Algebras Groups Geom. 24 (2007) 485-495.\vspace{-.1 in}

\bibitem{J} Y.B. Jun and S.Z. Song; {\it Generalized fuzzy interior ideals in semigroups,} Inform Sci. 176 (2006) 3079-3093.\vspace{-.1 in}

\bibitem{K1} N. Kuroki; {\it On fuzzy semigroups,} Inform Sci., 53 (1991) 203-236.\vspace{-.1 in}

\bibitem{K2} N. Kuroki; {\it On fuzzy ideals and fuzzy bi-ideals in semigroups,} Fuzzy Sets Syst., 5 (1981) 203-215.\vspace{-.1 in}

\bibitem{K3} N. Kuroki; {\it Fuzzy semiprime ideals in semigroups,} Fuzzy Sets Syst., 158 (2004) 277-288.\vspace{-.1 in}

\bibitem{Mu} V. Murali; {\it Fuzzy points of equivalent fuzzy subsets,} Inform. Sci., 158(2004) 277-288.\vspace{-.1 in}

\bibitem{Pu} P.M. Pu and Y.M. Liu; {\it Fuzzy topology I, neighbourhood structure of a fuzzy point and Moore-Smith convergence,} J. Math. Anal. Appl., 76 (1980) 571-599.\vspace{-.1 in}

\bibitem{R} A. Rosenfeld; {\it Fuzzy groups,} J. Math. Anal. Appl., 35 (1971) 512-517. \vspace{-.1 in}

\bibitem{S} N.K. Saha; {\it On $\Gamma$-semigroup II,} Bull. Cal. Math. Soc., 79 (1987) 331-335.\vspace{-.1 in}

\bibitem{S1} S.K. Sardar and S.K. Majumder; {\it On fuzzy ideals in $\Gamma $-semigroups,} International Journal of Algebra, 3(16) (2009)775-784.\vspace{-.1 in}

\bibitem{S2} S.K. Sardar, S.K. Majumder and D. Mandal; {\it A note on characterization of prime ideals of $\Gamma $-semigroups in terms of fuzzy subsets,} Int. J. Contemp. Math. Sciences, 4(30)(2009) 1465-1472.\vspace{-.1 in}

\bibitem{S3} S.K. Sardar and S.K. Majumder; {\it A note on characterization of semiprime ideals of $\Gamma $-semigroups in terms of fuzzy subsets,} International Journal of Pure and Applied Mathematics, 3(56)(2009)451-457.\vspace{-.1 in}

\bibitem{S4} S.K. Sardar, B. Davvaz and S.K. Majumder; {\it A study on fuzzy interior ideals of $\Gamma $-semigroups,} (Accepted), Computers and Mathematics with applications.\vspace{-.1 in}

\bibitem{S5} S.K. Sardar, S.K. Majumder and S. Kayal; {\it On fuzzy bi-ideals and fuzzy quasi ideals in $\Gamma $-semigroups,} (Communicated).\vspace{-.1 in}

\bibitem{SC} M.K. Sen and S. Chattopadhyay; {\it Semidirect product of a monoid and a $\Gamma$-semigroup,} East-West J. Math., 6 (2004) 131-138.\vspace{-.1 in}

\bibitem{SSa} M.K. Sen and N.K. Saha; {\it Orthodox $\Gamma$-semigroups,} Internat. J. Math. Math. Sci., 13 (1990) 527-534.\vspace{-.1 in}

\bibitem{SS} M.K. Sen and N.K. Saha; {\it On $\Gamma$-semigroup I,} Bull. Cal. Math. Soc., 78 (1986) 180-186.\vspace{-.1 in}

\bibitem{Se} A. Seth; {\it $\Gamma$-group congruences on regular $\Gamma$-semigroups,} Internat. J. Math. Math. Sci., 15 (1992)103-106.\vspace{-.1 in}

\bibitem{Yu} X. Yuan, C. Zhang and Y. Ren; {\it Generalized fuzzy groups and many-valued implications,} Fuzzy Sets Systs., 138 (2003)205-211.\vspace{-.1 in}

\bibitem{Yu1} Y. Yunqiang and D. Xu; {\it The $(\in ,\in \vee q)$-fuzzy subsemigroups and ideals of an $(\in ,\in \vee q)$-fuzzy semigroup,} Southeast Asian Bull. of math., 33 (2009)391-400.\vspace{-.1 in}

\bibitem{Z} L.A. Zadeh; {\it Fuzzy Sets,} Information and Control, 8 (1965) 338-353.\vspace{-.1 in}
\end{thebibliography}
\end{document}